\documentclass[12pt]{amsart}

\usepackage{graphicx, amsmath, amsthm, amssymb, xcolor, fullpage, hyperref}

\newtheorem{theorem}{Theorem}[section]

\newtheorem{lemma}[theorem]{Lemma}
\newtheorem{corollary}[theorem]{Corollary}

\newtheorem*{thmno}{Theorem}
\newtheorem*{corno}{Corollary}
\newtheorem*{remark}{Remark}
\newtheorem*{propno}{Proposition}

\numberwithin{equation}{section}
\numberwithin{theorem}{section}

\title{Antiquantum $q$-series identities \\and mock theta functions}
\author{Amanda Folsom and David Metacarpa}
\address{Department of Mathematics, Amherst College, Amherst, MA 01002, USA}
\email{afolsom@amherst.edu}
\email{dmetacarpa24@amherst.edu}
 \date{}
\thanks{\emph{Acknowledgments.} The authors are grateful for support from the National Science Foundation RUI Grant DMS-2200728 (PI=first author). The authors thank the anonymous referee for their helpful comments on an earlier version of the paper.}

\begin{document}

 \maketitle
 
\begin{center}\emph{Dedicated to  Krishnaswami Alladi, the Founding Editor in Chief of the\\ Ramanujan Journal, for his inspiration, commitment, and service.} \end{center}

 \begin{abstract}  
 Ramanujan's original definition of  mock theta functions from 1920 involves their asymptotic behaviors at roots of unity on the boundary of the disk of convergence $|q|<1$.  More recently this topic has  been  related by several authors, including the first author with Ono and Rhoades in 2013, to  quantum modular forms,  first defined in 2010  by Zagier.  In 2021, Lovejoy  defined and studied related quantum $q$-series identities, which do not hold as equalities between power series inside the disk $|q|<1$ but which do hold on   dense subsets of roots of unity on the boundary.  Inspired by this, in our prior joint work from 2024 we further studied quantum $q$-series identities as related to mock theta functions  and quantum modular forms;  we also defined and studied antiquantum $q$-series identities, between series which are equal  inside the disk $|q|<1$ but which hold at dense sets of roots of unity on the boundary for which one of the series diverges and is unnaturally truncated.   Here, building from our previous work, we establish antiquantum $q$-series identities for all of Ramanujan's third order mock theta functions.  We deduce these results in part by establishing and applying more general    identities  which are also of independent interest, and by using the theory of modular eta-quotients.  
     \end{abstract}
 
\section{Introduction}\label{sec_intro}

Understanding $q$-series and mock theta functions at roots of unity,  on the boundary of the  disk of convergence $|q|<1$, is a problem of interest.  For example, Ramanujan's original definition  of mock theta functions involves their asymptotic closeness   to modular theta functions at root of unity singularities, as illustrated in the following radial limit asymptotic for Ramanujan's third order mock theta function 
$$f(q):= \sum_{n=0}^\infty \frac{q^{n^2}}{(-q;q)_n^2}.  $$
Towards   (reduced) even order     root  of unity $\zeta_{2k}^h$ singularities \cite{BR, FOR},
\begin{align}\label{eqn_frl}\lim_{q\to \zeta_{2k}^h} \left(f(q) - (-1)^k(q;q^2)_\infty  \bigg(\sum_{n=-\infty}^\infty (-1)^n q^{n^2}\bigg) \right)= O(1).\end{align}
Here and throughout we use the notation $\zeta_N:=e^{2\pi i /N} (N\in\mathbb N),$ and say a root of unity $\zeta_b^a$ (or a fraction $a/b$) is \emph{reduced} if $b\in\mathbb N, a\in\mathbb Z,$ and $\gcd(a,b)=1$.  We also adopt the standard notation for the $q$-Pochhammer symbol 
$(a;q)_n := \prod_{j=0}^{n-1}(1-aq^j),$ for  $n\in \mathbb N_0 \cup \{ \infty \}$.  The function 
$(q;q^2)_\infty  (\sum_{n=-\infty}^\infty (-1)^n q^{n^2})$ appearing in \eqref{eqn_frl} is a modular form up to multiplication by $q^{-\frac{1}{24}}$, after specializing $q=e^{2\pi i \tau}, \tau \in \mathbb H:= \{\tau \in \mathbb C \colon \Im(\tau)>0\},$ the standard modular variable.  Moreover, the implied $O(1)$ constants on the right-hand side of \eqref{eqn_frl} are explicitly established in \cite[Thm.\ 1.1]{FOR} by the first author and Ono and Rhoades (via the more general result  \cite[Thm.\ 1.2]{FOR}), and are shown to be  special values of \emph{quantum modular forms}.   Loosely speaking, after Zagier \cite{Zagier}, a quantum modular form is a complex-valued function defined on $\mathbb Q$ (as opposed to the modular domain $\mathbb H$), which transforms like a modular form  under the action of $\textnormal{SL}_2(\mathbb Z)$ (or suitable subgroup) on $\mathbb Q$ but up to the addition of  analytic error functions in $\mathbb R$.  See \cite{Zagier} for details, or \cite{BFOR} and references therein for more on quantum modular forms and radial limits, both of which have been research topics of  interest in recent years. 

In \cite{LovejoyQ}, Lovejoy defined the related concept of \emph{quantum $q$-series identities}, which do not hold as equalities between series inside the unit disk in the clasical sense, but do hold on dense subsets of the boundary, namely, at roots of unity.  Earlier examples, prior to the  terminology, were given  by Cohen 
\cite{Cohen}, Bryson-Ono-Pitman-Rhoades \cite{BOPR}, and the first author with Ki, Truong Vu and Yang \cite{FKTY}.   In \cite{LovejoyQ} Lovejoy elegantly shows how to deduce these and many other quantum $q$-series identities using  $q$-series transformations, some of which involve Ramanujan's mock or false theta functions, including the following \cite[(1.28)]{LovejoyQ}
\begin{align}\label{eqn_ljnu}-\sum_{n=0}^\infty (-q;q^2)_n q^n = \sum_{n=0}^\infty (q^{-2};q^{-4})_n q^{-2n-1},\end{align}
which holds at reduced roots of unity $q=\zeta_{k}^h$ with $k\equiv 2\pmod{4}$, and which does \emph{not} hold in general as a  power series identity for $|q|<1$. The function 
\begin{align}\label{def_nutil0} \widetilde{\nu}(q) := \sum_{n=0}^\infty (q;q^2)_n(-1)^n q^n\end{align} appearing on the left-hand side of \eqref{eqn_ljnu} (with $q\mapsto -q$) is a known alternate $q$-hypergeometric series representation for the  third order mock theta function  
$\nu(q):= \sum_{n=0}^\infty q^{n^2+n}/(-q;q^2)_{n+1}$ inside the disk; that is, $\nu(q) = \widetilde{\nu}(q)$ for $|q|<1$ \cite{Fine}.

In our earlier joint work \cite{FM}, we further study quantum $q$-series,  as related to mock theta functions, $q$-hypergeometric series, radial limits, and quantum modular forms.  In the last section of \cite{FM}, we also define and investigate what we call \emph{antiquantum $q$-series identities,} inspired by and to complement Lovejoy's notion of quantum $q$-series \cite{LovejoyQ} (discussed above).   Our antiquantum $q$-series identities are between series which converge and are equal to each other inside the disk $|q|<1,$ but our identities hold at  dense  sets of roots of unity on the boundary for which one of the series diverges and is ``unnaturally" truncated.  To illustrate this, we restate here our antiquantum results from \cite{FM}, and begin with our general  identity  \cite[Prop.\ 5.1]{FM}, as  stated in the next theorem  which holds for reduced roots of unity $q=\zeta_k^h$.

\begin{thmno}[\cite{FM}, Prop.\ 5.1]
    For reduced roots of unity $q=\zeta_k^h$, we have that  $$ \sum_{n=0}^\infty \frac{(-b)^n q^{n^2}}{(bq;q)_n(-q;q)_n} = \frac{2q(1-b)}{(2-b^k)}\sum_{n=0}^{k-1}(b^{-1}q^2;q^2)_n (bq)^n.$$
\end{thmno}  This result reveals that  at 
  $k$th roots of unity $q$, an unnatural truncation of the series on the right-hand side (meaning when evaluated at $k$th roots of unity $q$, the corresponding infinite series does not in general equal its truncation at the $k-1$ summand) multiplied by a certain factor equals the series on the left-hand side. 

  \begin{remark}\textnormal{
 As is common with $q$-hypergeometric identities in the literature (see, e.g., \cite{Fine, GR, LovejoyQ}), we may state our results without enforced conditions on   parameters appearing (e.g.,  $b$) for maximum applicability, with the understanding that the identities may be used with any values such that the left- and right-hand sides simultaneously converge, or in other appropriate limiting situations.} 
  \end{remark}

Especially motivating our antiquantum identity in our  theorem stated above is its application  to two of Ramanujan's third order mock theta functions, namely to  
$$\psi(q) := \displaystyle \sum_{n=1}^\infty \frac{q^{n^2}}{(q;q^2)_n}, \ \ \ \ \ \ \  \phi(q) := \displaystyle\sum_{n=0}^\infty \frac{q^{n^2}}{(-q^2;q^2)_n}.$$   It is known \cite{Fine} that as $q$-series, inside the disk $|q|<1,$ these mock theta functions possess the alternate series representations  
$\psi(q) = \widetilde{\psi}(q)$ and $\phi(q) = \widetilde{\phi}(q)$, where

\begin{align*} 
\widetilde{\psi}(q)  := \sum_{n=0}^\infty (-q^2;q^2)_n q^{n+1}, \ \ \ \ \ \ \ 
\widetilde{\phi}(q)  := 1+\sum_{n=0}^\infty (q;q^2)_n  (-1)^n q^{2n+1}.
\end{align*}

Our earlier antiquantum result  from \cite{FM} for these mock theta functions  stated in the corollary below, ultimately deduced from \cite[Prop. 5.1]{FM} stated in the theorem above, reveals that when evaluated at  odd ordered roots of unity $q$, the mock theta functions $\psi(-q)$ and $\phi(-q)$  converge, and  are equal to  unnatural  truncations of their counterparts $\widetilde{\psi}(-q)$ and $\widetilde{\phi}(-q),$  which as infinite series do not converge at odd ordered roots of unity, and up to additional multiplication by and addition of constants. 

\begin{corno}[\cite{FM}, Cor.\ 5.2]
    For reduced roots of unity $q=\zeta_k^h$ with $k$ odd, we have that 
   $$\psi(-q) = \tfrac13 \widetilde{\psi}_{[k]}(-q)$$ and  $$\phi(-q) = \tfrac23 + \tfrac13 \widetilde{\phi}_{[k]}(-q).$$
\end{corno}

As in \cite{FM}, here and throughout we  use the notation $S_{[k]}$ to stand for the truncation of a series $S:= \sum_{n=0 }^\infty  a_n$ as follows
\begin{align}\label{def_trunc}   S_{[k]} := \sum_{0\leq n \leq k-1} a_n,\end{align}  so that $\lim_{k\to \infty} S_{[k]} = S.$  When $S$ is a function of the form  $S(x) := \sum_{n\geq 0} a_n(x)$, we will write $S_{[k]}(x)$ for $(S(x))_{[k]}$ (for ease of notation).  We extend the notation in the obvious way to series   of multiple variables. 

The purpose of this paper is to extend our initial study of antiquantum $q$-series identities in \cite{FM} described above.  In particular, we establish an additional more general    identity, and ultimately establish antiquantum $q$-series identities for all of Ramanujan and Watson's  mock theta functions of order three. Before doing so, we  first state our more general   identity, which complements \cite[Prop.\ 5.1]{FM} stated in the theorem above.  

In Theorem  \ref{thm_bid}, Corollary \ref{cor_nu} and Theorem \ref{thm_antimock3} below, $q=\zeta_k^h$, a reduced $k$th order root of unity; additional conditions on $k$ are given when relevant.

\begin{theorem}\label{thm_bid}
For reduced roots of unity $q=\zeta_k^h$ with $k\equiv 0 \pmod{4}$, we have that 
\begin{align}\notag
  \sum_{n=0}^{k-1}(b;q^2)_n  (-b^{-1}q^2)^n   & = p_{h,k}(b,q) (-1)^h (\pm b) (1-b^{-k/2}-(-1)^h) \\  \label{eqn_bidthm} &\hspace{.25in}\times 
\left(\sum_{n=0}^\infty \frac{q^{n^2+n}}{(\pm b^{-1}q^2;q^2)_{n+1}}- \left(\sum_{\ell=0}^\infty q^{\ell^2+\ell}\right)
\prod_{m=0}^\infty \frac{1}{1-b^{-2}q^{4m+4}}\right). \end{align}
\end{theorem}\noindent The product $p_{h,k}(b,q)$ is explicitly defined in \eqref{def_pk}.  

A corollary, albeit not an immediate one, of Theorem \ref{thm_bid} is the following antiquantum result (in \eqref{eqn_nuanti}) for the third order mock theta function $\nu$. \begin{corollary}\label{cor_nu} For reduced roots of unity $q=\zeta_k^h$ with   $k\equiv 0 \pmod{4},$ we have  that 
\begin{align}\label{eqn_nuanti}   \nu(-q) = -\tfrac13 \widetilde{\nu}_{[k]}(-q).\end{align} \end{corollary} \noindent We recall that the function $\widetilde{\nu}$ is defined in \eqref{def_nutil0}.  The antiquantum nature of the identity in \eqref{eqn_nuanti} is owed to the facts that inside the disk $|q|<1,$ $\nu(-q) = \widetilde{\nu}(-q)$ (as mentioned above \cite{Fine}), and (as infinite series) the function $\nu(-q)$ converges for reduced roots of unity  $q=\zeta_k^h$ with $k\equiv 0 \pmod{4},$ while $\widetilde{\nu}(-q)$ does not.  

To prove \eqref{eqn_nuanti}, we set  
 $b=-q = - \zeta_k^h$ with $k\equiv 0 \pmod{4}$ in \eqref{eqn_bidthm} from Theorem \ref{thm_bid} in the $+$ case; we also  apply  Lemma \ref{lem_modv} (see Section \ref{sec_eta}) which uses the theory of modular eta-quotients  to show that the 
function $$\left(\sum_{\ell=0}^\infty q^{\ell^2+\ell}\right)\prod_{m=0}^\infty \frac{1}{1-q^{4m+2}} = q^{-\frac13}
\frac{\eta^3(4\tau)}{\eta^2(2\tau)}
$$ appearing in \eqref{eqn_bidthm}
vanishes at  reduced roots of unity $q=\zeta_k^h$ with $k\equiv 0 \pmod{4}$; or equivalently, when viewed as an eta-quotient as a function of $\tau$ with $q=e^{2\pi i \tau},$ it vanishes at cusps $h/k$ (reduced) with $k\equiv 0 \pmod{4}.$    Here and throughout, Dedekind's $\eta$-function, a well-known modular form of weight $1/2,$ is given by
\begin{align}\label{def_eta}\eta(\tau):=q^{\frac{1}{24}}\prod_{n=1}^\infty (1-q^n),\end{align}  with $q=e^{2\pi i \tau},$ the standard modular variable.

We generalize our antiquantum identities for the third order mock theta functions $\nu$ in  Corollary \ref{cor_nu}, and $\psi$ and $\phi$ in \cite[Cor.\ 5.2]{FM} (also stated above), using Theorem \ref{thm_bid}, and results from our earlier work \cite{FM} and from Lovejoy's  \cite{LovejoyQ} (see also Section \ref{sec_FM}), to ultimately deduce antiquantum results for all of the  third order mock theta functions of Ramanujan and Watson \cite{BR, Watson}
$$\begin{array}{lclclcl}
f(q)&:= & \displaystyle \sum_{n=0}^\infty \frac{q^{n^2}}{(-q;q)_n^2}, && 
\omega(q) &:=&   \displaystyle \sum_{n=0}^\infty \frac{q^{2n(n+1)}}{(q;q^2)_{n+1}^2},  \\ \ \\ 
\psi(q) &:=& \displaystyle \sum_{n=1}^\infty \frac{q^{n^2}}{(q;q^2)_n}, && \phi(q) &:=&  \displaystyle\sum_{n=0}^\infty \frac{q^{n^2}}{(-q^2;q^2)_n}, \\ \ \\
\nu(q) &:=&   \displaystyle\sum_{n= 0}^\infty \frac{q^{n^2+n}}{(-q;q^2)_{n+1}}, 
& & 
\chi(q) &:=&  \displaystyle\sum_{n=0}^\infty \frac{q^{n^2}}{\prod_{j=1}^n(1-q^j+q^{2j})}, \\ \ \\ 
\rho(q) &:=  & \displaystyle\sum_{n=0}^\infty \frac{q^{2n(n+1)}}{\prod_{j=0}^n (1+q^{2j+1}+q^{4j+2})}.
\end{array}$$ 
For completeness, we state the theorem below inclusive of our earier results from \cite[Cor.\ 5.2]{FM}    $\psi$ and $\phi,$ as well as Corollary \ref{cor_nu} for $\nu$, discussed above.  The companion functions $\widetilde{f}_a(q), \widetilde{\omega}_a(q), \dots,$ for the third order mock theta functions appearing in Theorem \ref{thm_antimock3} are explicitly defined in  Section \ref{sec_mock};  those for $\psi, \phi$ and $\nu$   are (also) given by $\widetilde{\psi}, \widetilde{\phi},$ and $\widetilde{\nu}$ above.

\begin{theorem}\label{thm_antimock3}  The following antiquantum $q$-series identities are true: 
 \renewcommand*{\arraystretch}{2}  $$ 
\begin{array}{lcll}   
\psi(-q) & = & \displaystyle \tfrac13 (\widetilde{\psi}_a)_{[k]}(-q)    & \text{for } k \equiv 1 \pmod{2},  \ \\
 \phi(-q)  &  =  &  \displaystyle \tfrac23 + \tfrac13 (\widetilde{\phi}_a)_{[k]}(-q)  &  \text{for } k\equiv 1 \pmod{2},  \\
 \nu(-q) & = & \displaystyle -\tfrac13 (\widetilde{\nu}_a)_{[k]}(-q)  &   \text{for } k\equiv 0 \pmod{4}, \\
\omega(q^2) & = & \displaystyle -\tfrac13 (\widetilde{\omega}_a)_{[k]}(q^2) 
&   \text{for } k\equiv 0\pmod{4}, \\
\rho(q^2) & = & \displaystyle -\tfrac13 (\widetilde{\rho}_a)_{[k]}(q^2)  
&   \text{for } k\equiv 0 \pmod{12}, \\
f(q) & = & \displaystyle \tfrac13 (\widetilde{f}_a)_{[k]}(q) 
 \\ &    =    & \displaystyle \tfrac43 + \tfrac13 (\widetilde{\mathfrak{f}}_a)_{[k]}(q) 
  &   \text{for } k\equiv 1 \pmod{2},\\
\chi(q) & = &  \tfrac13
(\widetilde{\chi}_a)_{[k]}(q)     
\\ &    = &   \frac13 + \frac13 (\widetilde{{X}}_a)_{[k]}(q)  
&   \text{for } k \equiv 3 \pmod{6}.       
\end{array}$$ 
\end{theorem}
\begin{remark}\textnormal{
For $k$  in arithmetic progressions other than those given in Theorem \ref{thm_antimock3}, with two possible exceptions noted below, either the mock theta function does not converge, or the corresponding identity  is  not antiquantum but an ordinary identity at roots of unity.  For example, $\psi(-q)$ does not converge if $k\equiv 2\pmod{4},$ and $\psi(-q) = \widetilde{\psi}_a(-q)$  for $k\equiv 0\pmod{4}.$  Numerical evidence suggests that for certain (additional) even values of $k$,   additional antiquantum identities for $\rho(q^2)$ and $\chi(q)$ similar to those given in the above theorem may  hold; however, they do not appear to  immediately follow from the proofs in this paper.}
\end{remark}
 \ \\
 {{\bf{Example.}} Here we illustrate our antiquantum $q$-series identities from Corollary \ref{cor_nu} and Theorem \ref{thm_antimock3} 
 with the third order mock theta function    $\nu$  as an example.  
 
In general, as $q$-series inside the disk, we have that  $$\nu(-q) = \widetilde{\nu}(-q),  |q|<1.   $$  
 On the other hand, for   reduced $k$th roots of unity $q=\zeta_k^h$  with $k\equiv 0 \pmod{4}$, we have that $\nu(-q) \neq \widetilde{\nu}(-q),  q=\zeta_k^h,\ k\equiv 0 \pmod{4}k$ as follows: with $\zeta_k^h$ reduced and $k\equiv 0 \pmod{4}$,  $\nu(-\zeta_k^h)$ converges (see Lemma \ref{lem_mockconv}), while 
$\widetilde{\nu}(-\zeta_k^h)$ does not converge.   However, if we unnaturally truncate the infinite series defining $\widetilde{\nu}(-q)$ at the $(k-1)$st term, and then multiply it by the constant $-\tfrac13$, we have by Corollary \ref{cor_nu} or Theorem \ref{thm_antimock3} that
 $\nu(-q) = -\tfrac13   \widetilde{\nu}_{[k]}(-q), q=\zeta_k^h, \ k\equiv 0 \pmod{4}.$  
 
 To illustrate this with an explicit numerical example, we  let $(h,k)=(3,20)$ so that $q=\zeta_{20}^3$.  Then \begin{align}\label{eq_nuex} \nu(-\zeta_{20}^3) := 
\sum_{n=0}^\infty \frac{\zeta_{20}^{3(n^2+n)}}{(\zeta_{20}^3;\zeta_{20}^6)_{n+1}}
\approx -0.466695  + 1.38771  i,\end{align}
while $$\widetilde{\nu}(-\zeta_{20}^3) :=  \sum_{n=0}^\infty (-\zeta_{20}^3;\zeta_{20}^6)_n\zeta_{20}^{3n} \text{ \ \ does not converge}.$$ On the other hand after multiplying by a constant and unnaturally truncating the infinite series defining $\widetilde{\nu}(-q)$ (which diverges at $q=\zeta_{20}^3$), we have that 
 $$-\tfrac13   \widetilde{\nu}_{[k]}(-\zeta_{20}^3):= -\tfrac13  \sum_{n=0}^{19} (-\zeta_{20}^3;\zeta_{20}^6)_n\zeta_{20}^{3n}  \approx -0.466695  + 1.38771  i,$$  the same value as in \eqref{eq_nuex} due to Corollary \ref{cor_nu} or Theorem \ref{thm_antimock3}.  \medskip
 }

The remainder of the paper is structured as follows.  In Section \ref{sec_prelim} we give some preliminaries on mock theta functions, $q$-hypergeometric series and quantum $q$-series, and modular eta-quotients.  In Section \ref{sec_proofs}, we provide proofs of Theorems \ref{thm_bid} and \ref{thm_antimock3}. 

\section{Preliminaries}\label{sec_prelim}
\subsection{Mock theta functions}\label{sec_mock} 
 
In \cite{Watson}, Watson proved relations for the third order mock theta functions, including the following:
\begin{align}
\label{eqn_wat1} 2\phi(-q)-f(q) & = f(q) + 4\psi(-q) = \vartheta_4(0;q)\prod_{n=1}^\infty (1+q^n)^{-1}, \\ \label{eqn_wat2}
4\chi(q) - f(q) &= 3\vartheta_4^2(0;q^3)\prod_{n=1}^\infty (1-q^n)^{-1}, \\  \label{eqn_wat3}
2\rho(q) + \omega(q) &= \tfrac34 q^{-\frac34} \vartheta_2^2(0;q^{\frac32}) \prod_{n=1}^\infty (1-q^{2n})^{-1}, \\ \label{eqn_wat4}
\nu(- q)- q \omega(q^2) &= \tfrac12 q^{-\frac14}\vartheta_2(0;q)\prod_{n=1}^\infty (1+q^{2n}),
\end{align}
where the Jacobi theta functions \cite[20.2]{NIST} are defined by 
\begin{align}\label{def_theta2}
    \vartheta_2(0;q) &:= \sum_{n=-\infty}^\infty q^{(n+\frac12)^2}, \\ \label{def_theta4}
    \vartheta_4(0;q) &:= \sum_{n=-\infty}^\infty (-1)^n q^{n^2}.
\end{align}
We will make use of these relations in Section \ref{sec_proofs}.

Next we define the following nine companion functions to the third order mock theta functions (defined in Section \ref{sec_intro}): \allowdisplaybreaks
$$ \renewcommand*{\arraystretch}{2}  \begin{array}{lclcllll} \displaystyle
\widetilde{f}(q)&:=& \displaystyle 4\sum_{n=0}^\infty (-q^2;q^2)_n (-1)^n q^{n+1} + 
\prod_{n=1}^\infty \frac{(1-q^n)^2}{(1-q^{2n}) (1+q^n)} 
&  =: & \widetilde{f}_a(q) + u(q) &&&\\  
\widetilde{{\mathfrak f}}(q) &:= & \displaystyle 2  - 2\sum_{n=0}^\infty (-q;q^2)_n(-1)^n q^{2n+1}  - \prod_{n=1}^\infty \frac{(1-q^n)^2}{(1-q^{2n}) (1+q^n)}  & 
 =: &  \widetilde{\mathfrak{f}}_a(q) - u(q), &&&\\  
\widetilde{\omega}(q) &:=&  \displaystyle  q^{-\frac12}  \sum_{n= 0}^\infty  (-q^\frac12;q)_n q^{\frac{n}{2}} -q^{-\frac12}\prod_{n=1}^\infty \frac{(1-q^{2n})}{(1-q^{2n-1})^2} & =: &  \widetilde{\omega}_a(q) - v(q), &&&\\ 
\widetilde{\psi}(q) &:=& \displaystyle \sum_{n=0}^\infty (-q^2;q^2)_n q^{n+1} & =:  & \widetilde{\psi}_a(q),&&&
\\ 
\widetilde{\phi}(q) &:=& \displaystyle 1+\sum_{n=0}^\infty (q;q^2)_n  (-1)^n q^{2n+1} 
& =: & \widetilde{\phi}_a(q),&&&
\\  
\widetilde{\nu}(q) &:=&  \displaystyle \sum_{n=0}^\infty (q;q^2)_n(-1)^n q^n
 & =: &  \widetilde{\nu}_a(q),
\end{array}$$ $$\renewcommand*{\arraystretch}{2}\begin{array}{lcl}
\widetilde{\chi}(q) 
& := & \displaystyle  \sum_{n=0}^\infty (-q^2;q^2)_n (-1)^n q^{n+1} + 
\frac14 \prod_{n=1}^\infty \frac{(1-q^n)^2}{(1-q^{2n}) (1+q^n)}+\frac34  
\prod_{n=1}^\infty \frac{(1-q^{3n})^4}{(1-q^{6n})^2 (1-q^n)}
\\  &  =:& \widetilde{\chi}_a(q) +\tfrac14 u(q) + w(q),
\\   \widetilde{X}(q)&:=&    \displaystyle  \frac12  - \frac12\sum_{n=0}^\infty (-q;q^2)_n(-1)^n q^{2n+1}  \\ &&\hspace{.5in}\displaystyle- \frac14\prod_{n=1}^\infty \frac{(1-q^n)^2}{(1-q^{2n}) (1+q^n)}+\frac34  
\prod_{n=1}^\infty \frac{(1-q^{3n})^4}{(1-q^{6n})^2 (1-q^n)}
\\  & =: &  \widetilde{X}_a(q) -\tfrac14 u(q) + w(q), \\ 
\widetilde{\rho}(q) &:=&\displaystyle  -\tfrac12 q^{-\frac12}  \sum_{n= 0}^\infty (-q^\frac12;q)_n q^{\frac{n}{2}} +\tfrac12 q^{-\frac12}\prod_{n=1}^\infty \frac{(1-q^{2n})}{(1-q^{2n-1})^2}  + \frac32\prod_{n=1}^\infty  \frac{(1-q^{6n})^2}{(1-q^{6n-3})^2(1-q^{2n})} \\  
 &=:& \widetilde{\rho}_a(q) +\tfrac12 v(q) + x(q). \end{array}
$$
\\ 
The infinite products in the above expressions are   
\begin{align}\label{def_uvwx} \renewcommand*{\arraystretch}{3}\begin{array}{lclclcl}  
    u(q) &:=& \displaystyle \prod_{n=1}^\infty \frac{(1-q^n)^2}{(1-q^{2n}) (1+q^n)}, & &w(q) & :=  & \displaystyle \frac34  
\prod_{n=1}^\infty \frac{(1-q^{3n})^4}{(1-q^{6n})^2 (1-q^n)},   \\
    v(q) &:=& \displaystyle q^{-\frac12}\prod_{n=1}^\infty \frac{(1-q^{2n})}{(1-q^{2n-1})^2}, & &
x(q)& := & \displaystyle \frac32\prod_{n=1}^\infty  \frac{(1-q^{6n})^2}{(1-q^{6n-3})^2(1-q^{2n})},\end{array}
\end{align}
from which the series $\widetilde{g}_a(q)$ (for the third order mock theta functions $g$) appearing in Theorem \ref{thm_antimock3} are implicitly defined.  That is, $$\widetilde{f}_a(q) :=  \displaystyle 4\sum_{n=0}^\infty (-q^2;q^2)_n (-1)^n q^{n+1}, \ \ \ \ \   \widetilde{\omega}_a(q) := q^{-\frac12}  \sum_{n= 0}^\infty  (-q^\frac12;q)_n q^{\frac{n}{2}},$$ etc. We note that for the third mock theta functions $f$ and $\chi$ that we have defined two companion functions (each), namely $\widetilde{f}$ and $\widetilde{\mathfrak f}$ (for $f$), and $\widetilde{\chi}$ and $\widetilde{X}$ (for $\chi$). (Each also gives rise to two series  $\widetilde{f}_a$ and $\widetilde{\mathfrak f}_a$ (for $f$), and $\widetilde{\chi}_a$ and $\widetilde{X}_a$ (for $\chi$).)
The functions $\widetilde{g}$ above provide alternate series representations for the third order mock theta functions inside the unit disk  $|q|<1.$  We write this as a lemma.

\begin{lemma}
For the seven third order mock theta functions, we have that for $|q|<1$,  
\begin{align*}
f(q) &=  \widetilde{f}(q) = \widetilde{\mathfrak f}(q), \qquad\qquad \omega(q) = \widetilde{\omega}(q), \qquad\qquad \psi(q) = \widetilde{\psi}(q), \qquad
\phi(q) = \widetilde{\phi}(q), \\  \chi(q) &= \widetilde{\chi}(q) = \widetilde{X}(q), \qquad\qquad
 \nu(q)  = \widetilde{\nu}(q), \qquad\qquad
\rho(q) = \widetilde{\rho}(q). 
\end{align*}
\end{lemma}

\begin{proof} For  the third order mock theta functions $\psi, \phi,$ and $\nu,$ these identities are known (\cite{Fine}; also see Section \ref{sec_intro}).  For the other third order mock theta functions $g,$ these identities follow from Watson's relations for the third order mock theta functions \eqref{eqn_wat1}-\eqref{eqn_wat4}, along with the following known identities for the Jacobi theta functions \eqref{def_theta2}-\eqref{def_theta4} (see e.g.\ \cite[Theorem 2.8]{Andrews})

\begin{align} \label{eqn_jacids}
 \vartheta_2(0;q)  = 
2 q^{\frac14} \prod_{n=1}^\infty \frac{(1-q^{4n})}{(1-q^{4n-2})},
\qquad\qquad
  \vartheta_4(0;q)  =  \prod_{n=1}^\infty \frac{(1-q^n)^2}{(1-q^{2n})} 
\end{align}
and some simplifying.
\end{proof}

\begin{lemma}\label{lem_mockconv}
   The third order mock theta functions converge at the corresponding roots of unity given in Theorem \ref{thm_antimock3}.
\end{lemma}

\begin{proof}
    We prove convergence of $\nu(-q)$ at the desired roots of unity, noting that the proofs for the other six third order mock theta functions are similar.  In what follows, we establish the stronger result that $\nu(-q)$ converges for roots of unity $q=\zeta_k^h$ with $k\equiv 0 \pmod{2}$.
    Using that $|q^{n^2 + n}| = 1$ as $q$ is a root of unity, we observe that
    \begin{equation*}
        |\nu(-q)| \leq \sum_{n = 0}^\infty \frac{1}{|(q; q^2)_{n+1}|}.
    \end{equation*}
    Rewriting $n = s + m(k/2)$, where $0 \leq s \leq (k/2) - 1$ and $m \geq 0$, we have that the right-hand side above is  equal to 
    \begin{equation*}
        \sum_{s = 0}^{\frac{k}{2} - 1}\sum_{m \geq 0} \frac{1}{|(q; q^2)_{s+m(k/2)+1}|}.
    \end{equation*}
    For fixed $s$ we see that
    \begin{equation*}
        \sum_{m \geq 0} \frac{1}{|(q; q^2)_{s+m(k/2)+1}|} = \sum_{m \geq 0} \frac{1}{|(1 - q^{-\frac{k}{2}})^m(q; q^2)_{s+1}|} \\
        = \sum_{m \geq 0} \frac{1}{2^m|(q; q^2)_{s+1}|} \\
        = \frac{2}{|(q; q^2)_{s+1}|}.
    \end{equation*}
  Here we have used the identity $(xq;q)_{s+m\kappa} = (1 - x^\kappa)^m(xq;q)_s$ which holds for $\kappa$th roots of unity $q$, with 
$q \mapsto q^2, x = q^{-1}, s \mapsto s+1, \kappa \mapsto k/2.$
Hence $\nu(-q)$ converges when $q$ is a (reduced) $k$th root of unity with $k \equiv 0 \pmod{2}$.
\end{proof}

\subsection{Quantum $q$-series}\label{sec_FM}
Here we recall certain quantum $q$-series results from \cite{FM} and \cite{LovejoyQ}, which we later make use of in Section \ref{sec_proofs}.   
 
As in \cite{FM} we define for $r\in\mathbb N$ the $q$-hypergeometric series 
\begin{align}\notag\phi_r &
\left(\begin{array}{cccc}a_1 & a_2 & \cdots & a_r \\
b_1 & b_2 & \cdots & b_r \end{array};q;t\right) \\\label{def_phir} & := {}_{r+1} \phi_r\left(\begin{array}{ccccc}a_1q & a_2q & \cdots & a_r q & q \\
b_1q & b_2q & \cdots & b_rq \end{array};q;t\right)  = \sum_{n=0}^\infty \frac{(a_1q;q)_n \cdots (a_r q;q)_n}{(b_1q;q)_n \cdots (b_r q;q)_n}t^n,
\end{align}
and the quantities 
$$u_{r,k}(\vec{a},\vec{b},t)
:= \frac{(1-a_1^k)\cdots (1-a_r^k)}{(1-b_1^k)\cdots(1-b_r^k)}t^k,
$$
$$\delta_{r,k}(\vec{a},\vec{b},t):= \left(1-u_{r,k}(\vec{a},\vec{b},t)\right)^{-1},$$
and
$$c_{r,k}(\vec{a},\vec{b},t)  := \delta_{r,k}(\vec{a},\vec{b},t) \frac{u_{r,k}(\vec{a},\vec{b},t)}{u_{r,1}(\vec{a},\vec{b},t)},$$
using the notation $\vec{x}=\vec{x}_r := (x_1,x_2,\dots,x_r)$.
In \cite[Prop.\ 2.1]{FM} we establish the following general quantum $q$-series identity.  
\begin{propno}[\cite{FM}, Prop.\ 2.1] Let $r,k \in \mathbb N$, and $q=\zeta_k^h$ with $h/k$ reduced. Then  
\begin{align}\label{eqn_thmphi} \phi_r\left(\begin{array}{cccc}a_1 & a_2 & \cdots & a_r \\
b_1 & b_2 & \cdots & b_r \end{array};q;t\right) =  c_{r,k}(\vec{a},\vec{b},t) (\phi_r)_{[k]}\left(\begin{array}{cccc}b_1 & b_2 & \cdots & b_r \\
a_1 & a_2 & \cdots & a_r \end{array};q^{-1};t^{-1}\right).\end{align}
\end{propno} 

We also make use of the following two-variable polynomial identity due to Lovejoy \cite{LovejoyQ}.
\begin{propno}[\cite{LovejoyQ},  (2.45)] For reduced roots of unity $q=\zeta_k^h,$ we have that 
\begin{align}\label{eqn_lovejoyb23}
\sum_{n=0}^{k-1}(b;q)_n z^n = b^{k-1}\sum_{n=0}^{k-1}(b;q)_n(q/z;q)_n(z/b)^n q^{-n^2-n}.
\end{align}
\end{propno}

\subsection{Eta-quotients}
\label{sec_eta}

Eta-quotients, or $\eta$-quotients,   as mentioned in Section \ref{sec_intro}, are relevant to our antiquantum $q$-series identities in that they arise in both the proof of Theorem \ref{thm_bid} and the definitions of $u(q), v(q), w(q), x(q)$ \eqref{def_uvwx} found in our companion functions to the third order mock theta functions. By \cite[Definition 1.63]{Ono}, an $\eta$-quotient is a function of the form $$F(\tau) = \prod_{\delta | N} \eta(\delta\tau)^{r_\delta},$$ where $\eta(\tau)$ is Dedekind's $\eta$-function \eqref{def_eta}, $N \in\mathbb N$, and each $r_\delta \in \mathbb{Z}$.  For example, we may rewrite the infinite products in \eqref{def_uvwx} in terms of $\eta$-quotients as follows, with $q=e^{2\pi i \tau}, \tau \in \mathbb H$:
\begin{align}\label{eqn_uvwxeta} \renewcommand*{\arraystretch}{2.5} \begin{array}{lclclcl}
    u(q) &=& \displaystyle q^{\frac{1}{24}}\frac{\eta^3(\tau)}{\eta^2(2\tau)},&&  w(q) &=& \displaystyle \tfrac34  
q^{\frac{1}{24}} \frac{\eta^4(3\tau)}{\eta^2(6\tau)\eta(\tau)}, \\
    v(q) &=& \displaystyle q^{-\frac23} \frac{\eta^3(2\tau)}{\eta^2(\tau)},&& 
x(q) &=& \displaystyle \tfrac{3}{2} q^{-\frac23} \frac{\eta^4(6\tau)}{\eta^2(3\tau)\eta(2\tau)}.\end{array}
\end{align} In Lemma \ref{lem_etaquot} below, used in our proof of Theorem \ref{thm_antimock3} in Section \ref{sec_proofs}, we establish the vanishing of these functions at certain roots of unity $q=\zeta_k^h$, or equivalently, when viewed as functions of $\tau$, at certain cusps $\tau=h/k$. Our proof of Lemma \ref{lem_etaquot} below makes use of the following general theorem which explicitly gives the orders of vanishing at cusps of $\eta$-quotients.
\begin{propno}[\cite{Ono}, Theorem 1.65]\label{lem_ono}
    Let $c, d$ and $N$ be positive integers with $d \mid N$ and $\gcd(c, d) = 1$.  If $F(\tau)$ is an $\eta$-quotient such that
    \begin{align} \label{eqn_etacond1}
        \sum_{\delta \mid N}\delta r_\delta &\equiv 0 \pmod{24}, \\ \label{eqn_etacond2}
        \sum_{\delta \mid N} \frac{N}{\delta}r_\delta &\equiv 0 \pmod{24},
    \end{align}
then the order of vanishing of $F(\tau)$ at the cusp $\frac{c}{d}$ is

    \begin{equation*}
        \frac{N}{24} \sum_{\delta \mid N} \frac{\gcd(d, \delta)^2 r_\delta}{\gcd(d, \frac{N}{d})d\delta}.
    \end{equation*}
\end{propno}

Using the above, we readily prove:

\begin{lemma}\label{lem_etaquot}
The following are true:
\begin{itemize}
    \item[i)] The $\eta$-quotient $\displaystyle \frac{\eta^3(\tau)}{\eta^2(2\tau)}$ vanishes at cusps $\tau=h/k$ (reduced) with $k\equiv 1 \pmod{2}.$
\\
\item[ii)] The $\eta$-quotient $\displaystyle \frac{\eta^3(2\tau)}{\eta^2(\tau)}$ vanishes at cusps $\tau=h/k$ (reduced) with $k\equiv 0 \pmod{2}.$
\item[iii)] The $\eta$-quotient $\displaystyle \frac{\eta^4(3\tau)}{\eta^2(6\tau)\eta(\tau)}$ vanishes at cusps $\tau=h/k$ (reduced) with $k\equiv 3 \pmod{6}.$ \\
\item[iv)] The $\eta$-quotient $\displaystyle \frac{\eta^4(6\tau)}{\eta^2(3\tau)\eta(2\tau)}$ vanishes at cusps $\tau=h/k$ (reduced) with $k\equiv 0 \pmod{6}.$
\end{itemize}
\end{lemma}

\begin{lemma} \label{lem_modv}
Let $k\in \mathbb N$ with $k \equiv 0 \pmod{4}$, and $h\in \mathbb Z$ such that $\gcd(h,k)=1.$ Then $$
\lim_{q\to \zeta_k^h} \left(\left(\sum_{\ell=0}^\infty q^{\ell^2+\ell}\right)
\prod_{m=0}^\infty \frac{1}{1-q^{4m+2}} \right) = 0.$$
\end{lemma}  

\begin{proof}[Proof of Lemma \ref{lem_etaquot}.]
    We prove ii) here, and note that the proofs for the other three assertions in the lemma are similar and  left to the reader.
    Define the $\eta$-quotient

    \begin{equation*}
        G(\tau) := \frac{\eta^3(12\tau)}{\eta^2(6\tau)}.
    \end{equation*}
     Consider the cusp $\frac{c}{d}$ (reduced) with  $d \equiv 0 \pmod{4}$, and let $N = 288d$.  It is not difficult to directly check (noting that $r_{12} = 3, r_6 = -2$, and $r_j = 0$ for all other $j \mid N$) that conditions \eqref{eqn_etacond1} and \eqref{eqn_etacond2} hold.  By the proposition above (\cite{Ono}, Theorem 1.65) we have that the order of vanishing of $G(\tau)$ at $\frac{c}{d}$ is 

    \begin{align}\notag 
\frac{N}{24} \sum_{\delta \mid N} \frac{\gcd(d, \delta)^2 r_\delta}{\gcd(d, \frac{N}{d})d\delta} &=   12d\left(\frac{3\gcd(d, 12)^2}{12d\gcd(d, 288)} - \frac{2\gcd(d, 6)^2}{6d\gcd(d, 288)} \right) \\ \label{eqn_etaord}
        &= \frac{3\gcd(d, 12)^2 - 4\gcd(d, 6)^2}{\gcd(d, 288)}.
    \end{align}
    There are two cases to consider. \smallskip

    \textbf{Case 1:}  $d \equiv 0 \pmod{12}$.  
    Then $\gcd(d, 12) = 12$ and $\gcd(d, 6) = 6,$ so 
    \begin{equation*}
        3\gcd(d, 12)^2 - 4\gcd(d, 6)^2 = 3\cdot 144 - 4\cdot 36 = 288 > 0.
    \end{equation*}
    
    \textbf{Case 2:}  $d \not\equiv 0 \pmod{12}$.  
    Then $\gcd(d, 12) = 4$ and $\gcd(d, 6) = 2$, so
    \begin{equation*}
        3\gcd(d, 12)^2 - 4\gcd(d, 6)^2 = 3\cdot 16 - 4\cdot 4 = 32 > 0.
    \end{equation*}

    Therefore, we see that \eqref{eqn_etaord} is positive, and hence $G(\tau)$ vanishes at cusps $\frac{c}{d}$ (reduced) with $d \equiv 0 \pmod{4}$.

    Finally, we show that
    \begin{equation*}
        F(\tau) := \frac{\eta^3(2\tau)}{\eta^2(\tau)} = G(\tau/6)
    \end{equation*}
    vanishes at cusps $\tau = \frac{h}{k}$ with $k\equiv 0 \pmod{2}$.  We  write  $F(\frac{h}{k}) = G(\frac{h'}{k'}),$ 
    where $\frac{h'}{k'} = \frac{h}{6k}$, and $\frac{h'}{k'}$ is reduced.  Thus, the  vanishing of $F$ at cusps $h/k$ with $k\equiv 0 \pmod{2}$ will follow from the vanishing of $G$ at cusps $h'/k'$ with $4\mid k'$ established above, after we show that $k' \equiv 0 \pmod{4}$.  To see this, observe that 
    \begin{equation*}
        k' = \frac{6k}{\gcd(h, 6k)}.
    \end{equation*}
    Since $\gcd(h, k) = 1$, we have that $2 \nmid \gcd(h, 6k)$.  Additionally, since $2 \mid k$, we see that $4 \mid 6k$.  Therefore $k' \equiv 0 \pmod{4}$ as wanted.
\end{proof}

\begin{proof}[Proof of Lemma \ref{lem_modv}] Using the definition of $\eta$ \eqref{def_eta} and some simplifying we write (with $q=e^{2\pi i \tau}$)
    \begin{equation*}
        \prod_{m=0}^\infty \frac{1}{1-q^{4m+2}} = q^{-\frac{1}{12}}\frac{\eta(4\tau)}{\eta(2\tau)}.  
    \end{equation*}
Further, using the Jacobi triple product identity \cite[Theorem 2.8]{Andrews} and the fact that $\sum_{\ell = 0}^\infty q^{\ell^2 + \ell} = \frac{1}{2}\sum_{\ell = -\infty}^\infty q^{\ell^2 + \ell}$, we write (again with $q=e^{2\pi i \tau}$)
    \begin{equation*}
        \sum_{\ell = 0}^\infty q^{\ell^2 + \ell} = q^{-\frac14} \frac{\eta^2(4\tau)}{\eta(2\tau)}.
    \end{equation*}
Thus,
    \begin{equation}\label{eqn_13proofeta}
        \left(\sum_{\ell=0}^\infty q^{\ell^2+\ell}\right) \prod_{m=0}^\infty \frac{1}{1-q^{4m+2}} = q^{-\frac13}\frac{\eta^3(4\tau)}{\eta^2(2\tau)},
    \end{equation}
    and by Lemma \ref{lem_etaquot}, we see that the right-hand side of (\ref{eqn_13proofeta}) vanishes when $2\tau = c/d$ (reduced) with $d \equiv 0 \pmod 2$, i.e., at (reduced) cusps $\tau = h/k$ with $k \equiv 0 \pmod 4$.
\end{proof}

\section{Proofs of Theorem \ref{thm_bid} and Theorem \ref{thm_antimock3}}\label{sec_proofs}
\subsection{Proof of Theorem \ref{thm_bid}}
Let $q=\zeta_k^h$ be a reduced root of unity as in the statement of the theorem.  For any $b$, we
  have that 
\begin{align}\notag
     \sum_{n=0}^{k-1}(b;q^2)_n (-b^{-1}q^2)^n  
    &= \Big(\sum_{n=0}^{\frac{k}{2}-1} + \sum_{n=\frac{k}{2}}^{k-1} \Big)(b;q^2)_n(-b^{-1}q^2)^n \\ \label{eqn_bid1} &= \left(1+(-b^{-1})^{\frac{k}{2}}(1-b^{\frac{k}{2}})\right)
    \sum_{n=0}^{\frac{k}{2}-1} (b;q^2)_n(-b^{-1}q^2)^n,
\end{align} 
where we have changed the index of summation in the second sum on the right hand side of the first line above, and have used that for $k$th roots of unity $q$, $(xq;q)_{s+mk} = (1-x^k)^m(xq;q)_s$, $(s,m\in\mathbb N_0).$

 Next we apply Lovejoy's \cite[(2.45)]{LovejoyQ}, given in \eqref{eqn_lovejoyb23}, with $q\mapsto q^2,$ and $z=-b^{-1}q^2$, noting that $q^2$ is a reduced root of unity of order $k/2,$ to obtain   that \eqref{eqn_bid1} equals
\begin{align}\notag &
\left(1+(-b^{-1})^{\frac{k}{2}}(1-b^{\frac{k}{2}})\right)b^{\frac{k}{2}-1}\sum_{n=0}^{\frac{k}{2}-1}(b;q^2)_n(-b;q^2)_n(-b^{-2}q^2)^n q^{-2n^2-2n} \\ \notag
&= \left(1+(-b^{-1})^{\frac{k}{2}}(1-b^{\frac{k}{2}})\right)b^{\frac{k}{2}-1}\sum_{n=0}^{\frac{k}{2}-1}(b^2;q^4)_n(-b^{-2})^n q^{-2n^2}
 \\ \notag
&= \left(1+(-b^{-1})^{\frac{k}{2}}(1-b^{\frac{k}{2}})\right)b^{\frac{k}{2}-1}\sum_{n=0}^{\frac{k}{2}-1}(b^{-2};q^{-4})_n  q^{-2n}\\ \label{eqn_bid2}
&= p_{h,k}(b,q) \sum_{n=0}^{\frac{k}{4}-1}(b^{-2};q^{-4})_n  q^{-2n},
\end{align}
where
\begin{align}\label{def_pk} p_{h,k}(b,q) := \left(1+(-b^{-1})^{\frac{k}{2}}(1-b^{\frac{k}{2}})\right)b^{\frac{k}{2}-1}\left(1+(-1)^h(1-b^{-\frac{k}{2}})\right).\end{align}
Above, we have used that
$(x;q^{-1})_n = (-x)^{n} q^{-\frac{n(n-1)}{2}}(x^{-1};q)_n$.

Next we apply our \cite[Proposition 2.1]{FM}, also given in Section \ref{sec_FM}, with $r=1, a_1=0, b_1=b^{-2}q^4, t =  q^2$, $k\mapsto \frac{k}{4}$ and $q\mapsto q^4,$ which shows that  \eqref{eqn_bid2} equals
\begin{align}\label{eqn_bid3}
  p_{h,k}(b,q)c^{-1}_{1,\frac{k}{4}}(0,b^{-2}q^4,q^2) \phi_1\left( \begin{array}{c}
0 \\ b^{-2}q^4
\end{array};q^4;q^2\right),
\end{align}
where \begin{align*}
    c_{1,\frac{k}{4}}(0,b^{-2}q^4,q^2) &= 
     \delta_{1,\frac{k}{4}}(0,b^{-2}q^4,q^2) \frac{ u_{1,\frac{k}{4}}(0,b^{-2}q^4,q^2) }{ u_{1,1}(0,b^{-2}q^4,q^2) } \\
     &= \left(1-\frac{q^{k/2}}{1-b^{-k/2}}\right)^{-1}
\frac{q^{k/2}}{1-b^{-k/2}}\cdot \frac{1-b^{-2}q^4}{q^2} \\
&=   \frac{(-1)^h}{1-b^{-k/2}-(-1)^h}\cdot\frac{1-b^{-2}q^4}{q^2}.
\end{align*}  

The following is equivalent to a $q$-hypergeometric identity of Fine,  upon replacing $q\mapsto q^2$ and $b\mapsto \pm b^{-1}q^2$ in \cite[(8.4)]{Fine}:
\begin{align*} \sum_{n=0}^\infty \frac{q^{n^2+n}}{(\pm b^{-1}q^2;q^2)_{n+1}} = 
   \left(\sum_{\ell=0}^\infty q^{\ell^2+\ell}\right)
\prod_{m=0}^\infty \frac{1}{1-b^{-2}q^{4m+4}} + \frac{\pm b^{-1} q^2}{1-b^{-2}q^4}\sum_{n=0}^\infty \frac{q^{2n}}{(b^{-2}q^8;q^4)_n}.\end{align*}Equivalently (also recalling definition  \eqref{def_phir}) we have that
\begin{align}\label{eqn_bf}
   \phi_1\left( \begin{array}{c}
0 \\ b^{-2}q^4
       \end{array};q^4; q^2\right)
    = \frac{1-b^{-2}q^4}{\pm b^{-1}q^2} \left(\sum_{n=0}^\infty \frac{q^{n^2+n}}{(\pm b^{-1}q^2;q^2)_{n+1}}- \left(\sum_{\ell=0}^\infty q^{\ell^2+\ell}\right)
\prod_{m=0}^\infty \frac{1}{1-b^{-2}q^{4m+4}}\right).
\end{align}  Using \eqref{eqn_bf} with \eqref{eqn_bid3} proves that with $q=\zeta_k^h$ a reduced $k$th root of unity, we have 
\begin{align}\notag 
  \sum_{n=0}^{k-1}(b;q^2)_n  (-b^{-1}q^2)^n    & = \ p_{h,k}(b,q) (-1)^h (\pm b) (1-b^{-k/2}-(-1)^h) \\ &\hspace{.3in} \times 
\left(\sum_{n=0}^\infty \frac{q^{n^2+n}}{(\pm b^{-1}q^2;q^2)_{n+1}}- \left(\sum_{\ell=0}^\infty q^{\ell^2+\ell}\right)
\prod_{m=0}^\infty \frac{1}{1-b^{-2}q^{4m+4}}\right) \label{eqn_bid4}\end{align} as wanted. \hfill\qed

\subsection{Proof of Theorem \ref{thm_antimock3}}
The antiquantum identities stated in Theorem \ref{thm_antimock3} for the third order mock theta functions $\psi$ and $\phi$ are established in \cite[Cor.\ 5.2]{FM}, in part by using a special case of our more general result \cite[Prop.\ 2.1]{FM}  (see also Section \ref{sec_intro} and Section \ref{sec_FM}).  Similarly, the antiquantum identity stated in Theorem \ref{thm_antimock3} for the third order mock theta function $\nu$ is established as a corollary to Theorem \ref{thm_bid} in Corollary \ref{cor_nu}, and its proof is given in Section \ref{sec_intro}.  

To prove the remaining antiquantum identities stated in Theorem \ref{thm_antimock3} for the third order mock theta functions $\omega, \rho, f$ and $\chi$, we use Watson's relations for the third order mock theta functions \eqref{eqn_wat1}-\eqref{eqn_wat4}, combined with our aforementioned antiquantum identities for $\psi, \phi$ and $\nu$, as well as the Jacobi theta identities \eqref{eqn_jacids}, our Lemma \ref{lem_etaquot}, and some simplifying. We also use convergence of the third order mock theta functions appearing in Theorem \ref{thm_antimock3} at the relevant roots of unity (see Lemma \ref{lem_mockconv}).  

To this end, we provide a detailed proof of the stated antiquantum identity for the third order mock theta function $\omega$ below.  For brevity's sake, we leave the detailed proofs of the remaining antiquantum identities for the third order mock theta functions $\rho, f$ and $\chi$ to the reader, noting that they follow similarly, using the methods stated above.   

We use Watson's \cite{Watson} third order mock theta relation \eqref{eqn_wat4}
$$\nu(-q)-q\omega(q^2) = \tfrac12 q^{-\frac14}\vartheta_2(0;q)\prod_{r=1}^\infty (1+q^{2r}),$$
which is equivalent to

\begin{align}\notag \omega(q^2) &= q^{-1}\nu(-q) 
-\tfrac12 q^{-\frac54}\vartheta_2(0;q) \prod_{r=1}^\infty(1+q^{2r}) \\ \notag
&= q^{-1}\nu(-q) - q^{-1}\prod_{n=1}^\infty\frac{(1-q^{4n})}{(1-q^{4n-2})^2}\\ \notag 
&= q^{-1}\nu(-q) - v(q^2) 
\\
\notag
&= q^{-1}\sum_{n=0}^\infty (-q;q^2)_n q^n - v(q^2) \ \notag \\ &= \widetilde{\omega}_a(q^2) - v(q^2) \notag \\ \label{eqn_omaid} &= \widetilde{\omega}(q^2),\end{align}
for $|q|<1.$ Here we have used the Jacobi theta identity for $\vartheta_2$  in \eqref{eqn_jacids}, Euler's identity 
$(-q;q)_\infty = (q;q^2)_\infty^{-1}$ \cite[(1.2.5)]{Andrews},
 the definitions of $\widetilde{\omega}_a, v, \widetilde{\omega}$ and that $\nu(-q) = \widetilde{\nu}(-q)$ for $|q|<1.$ 

Next,  (using the definition of $\eta(\tau)$ in  \eqref{def_eta}) we re-write 
$$v(q) = q^{-\frac23} \frac{\eta^3(2\tau)}{\eta^2(\tau)},$$
for $q=e^{2\pi i \tau}, \tau \in \mathbb H.$  Combining this with part ii) of Lemma \ref{lem_etaquot} establishes that $v(q^2)$ vanishes at (reduced) roots of unity $\zeta_k^h$ with $k\equiv 0 \pmod{4}$.   It is not difficult to check (see Lemma \ref{lem_mockconv}   and its proof in Section \ref{sec_mock}) that $\omega(q^2)$ and $\nu(-q)$ converge at (reduced) roots of unity $\zeta_k^h$ with $k\equiv 0 \pmod{4}$.  Further, by the dominated convergence theorem, the values of these mock theta functions at such roots of unity equal 
their radial limits.  Hence, the above shows that $\omega(q^2) = q^{-1}\nu(-q)$ for $q=\zeta_k^h$ with $k\equiv 0 \pmod{4}.$   Combining this with the fact that $\widetilde{\omega}_a(q^2) = q^{-1}\widetilde{\nu}(-q),$ as well as $\nu(-q) = -\tfrac13 (\widetilde{\nu}_a)_{[k]}(-q)$ for $q=\zeta_k^h$ (reduced) with $k\equiv 0 \pmod 4$ from Corollary \ref{cor_nu}, we have that 

$$\omega(q^2) = - \tfrac13 (\widetilde{\omega}_a)_{[k]}(q^2)$$ for $q=\zeta_k^h$ (reduced) with $k\equiv 0 \pmod{4}$
as claimed.  

To conclude, we  elaborate on the  antiquantum nature of this identity by confirming that $\widetilde{\omega}_a(q^2)$ does not converge (as an infinite series) at roots of unity $q=\zeta_k^h$ with $k\equiv 0 \pmod{4}$, and is thus  unnaturally  truncated at the $k-1$ term in our identity, and multiplied by $-1/3$.  \hfill\qed


\begin{thebibliography}{99}
\bibitem{Andrews} G.\ E.\ Andrews, {\it The Theory of Partitions}, Encyclopedia of Mathematics and its Applications, vol.\ 2.
Cambridge University Press, Cambridge. (1984)

\bibitem{BR} B.\ C.\ Berndt and R.\ A.\ Rankin, {\it{Ramanujan: Letters and Commentary}} (American Mathematical Society, Providence, 1995). 

\bibitem{BFOR} K.\ Bringmann, A.\ Folsom, K.\ Ono, and L.\ Rolen, {\it{Harmonic Maass forms and mock modular forms: theory and applications,}} AMS Colloquium Publications 64.  American Mathematical Society, Providence, RI, 2018. 391pp. 

\bibitem{BOPR} J.\ Bryson, K.\ Ono, S.\ Pitman, and R.\ C.\ Rhoades, {\it{Unimodal sequences and quantum and mock modular forms}}, Proc.\ Natl.\ Acad.\ Sci.\ USA 109 (2012), no. 40, 16063--16067. 

\bibitem{Cohen} H.\ Cohen, {\it{$q$-identities for Maass waveforms}}, Invent.\ Math. 91 (1988), no.\ 3, 409-422. 

\bibitem{Fine} N.J. Fine, {\it Basic hypergeometric series and applications},   Mathematical Surveys and Monographs, vol.\ 27.  
American Mathematical Society, Providence, RI. (1988)

\bibitem{FKTY} A.\ Folsom, C.\ Ki, Y.\ N. Truong Vu, and B.\ Yang, {\it{``Strange" combinatorial quantum modular forms,}} J.\ Number Theory 170 (2017), 315--346.  

\bibitem{FM} A.\ Folsom and D.\ Metacarpa, {\it Quantum $q$-series and mock theta functions}, 
 Res.\ Math.\ Sci. 11 no. 41 (2024), 21pp.

\bibitem{FOR} A.\ Folsom, K.\ Ono, and R.\ C.\ Rhoades, {\it{Mock theta functions and quantum modular forms}}, Forum Math. Pi 1 (2013), 1--27.

\bibitem{GR} G.\ Gasper and G.\ Rahman, {\it{Basic Hypergeometric Series, volume 96 of Encyclopedia of Mathematica and its Applications, 2nd Ed., Cambridge University Press, Cambridge.  (2004).}}

\bibitem{LovejoyQ} J.\ Lovejoy, {\it Quantum $q$-series identities}, Hardy-Ramanujan J.\ 44 (2021), 61--73.

\bibitem{NIST} NIST Digital Library of Mathematical Functions, \url{http://dlmf.nist.gov}, Release 1.2.4 of 2025–03-15. F. W. J. Olver, A. B. Olde Daalhuis, D. W. Lozier, B. I. Schneider, R. F. Boisvert, C. W. Clark, B. R. Miller, and B. V. Saunders, eds.

\bibitem{Ono} K.\ Ono, {\it The Web of Modularity: Arithmetic of the Coefficients of Modular Forms and $q$-series}, Regional Conference Series in Mathematics, vol.\ 102.
American Mathematical Society, Providence, RI. (2004)


\bibitem{Watson} G.\ N.\ Watson, {\it The Final Problem: An Account of the Mock Theta Functions}, J.\ London Math.\ Soc.\ vol.\ s1-11 iss.\ 1 (1936), 55--80. 

\bibitem{Zagier} D.\ Zagier, {\it{Quantum modular forms}}, in Quanta of Maths, Clay Mathematics Proceedings 11 (American Mathematical Society, Providence, (2010), 659--675.

\end{thebibliography}
\end{document}